\RequirePackage{fix-cm}
\documentclass[smallextended]{svjour3}       
\smartqed  
\usepackage{amsmath}
\usepackage{amssymb}
\usepackage{dsfont}
\usepackage{amsfonts}
\usepackage{graphicx}
\RequirePackage[colorlinks,citecolor=blue,urlcolor=blue]{hyperref}

\journalname{}
\begin{document}

\title{Adiabatic reduction of a model of stochastic gene expression with jump Markov process}
\thanks{This research was supported by the Ecole Normale Superieure Lyon(ENS Lyon,France), the National Natural Science Foundation of China, and the Natural Sciences and Engineering Research Council of Canada.}

\titlerunning{Adiabatic reduction of jump Markov process}        

\author{Romain Yvinec        \and
       Changjing Zhuge \and
       Jinzhi Lei \and	
      Michael C. Mackey
}

\authorrunning{Yvinec, Zhuge, Lei, Mackey} 

\institute{Romain Yvinec \at
              Universit\'e de Lyon CNRS UMR 5208 Universit\'{e} Lyon 1, Institut Camille Jordan,
43 blvd. du 11 novembre 1918,  F-69622 Villeurbanne Cedex France \\
              Tel.: +33472431189\\
              \email{yvinec@math.univ-lyon1.fr}
\and
	  Changjing Zhuge \at
	      Zhou Pei-Yuan Center for Applied Mathematics, Tsinghua University, Beijing 100084, China
	      \email{zgcj08@mails.tsinghua.edu.cn}
           \and
           Jinzhi Lei \at
              Zhou Pei-Yuan Center for Applied Mathematics, Tsinghua University, Beijing 100084, China
              \email{jzlei@tsinghua.edu.cn}
	  \and
	   Michael C Mackey \at
	      Departments of Physiology, Physics \& Mathematics and Centre for
Applied Mathematics in Bioscience \& Medicine, McGill University, 3655 Promenade Sir William Osler, Montreal, QC, CANADA, H3G 1Y6
\email{michael.mackey@mcgill.ca}
}

\date{Received: date / Accepted: date}

\maketitle

\begin{abstract}
This paper considers adiabatic reduction in a model of stochastic gene expression with bursting transcription considered as a jump Markov process. In this model, the process of gene expression with auto-regulation is described by fast/slow dynamics. The production of mRNA is assumed to follow a compound Poisson process occurring at a rate depending on protein levels (the phenomena called bursting in molecular biology) and the production of protein is a linear function of mRNA numbers. When the dynamics of mRNA is assumed to be a fast process (due to faster mRNA degradation than that of protein) we prove that, with  appropriate scalings in the burst rate, jump size or translational rate, the bursting phenomena can be transmitted to the slow variable. We show that, depending on the scaling,  the reduced equation is either a stochastic differential equation with a jump Poisson process or a deterministic ordinary differential equation. These results are significant  because adiabatic reduction techniques seem to have not been rigorously justified for a stochastic differential system containing  a jump Markov process.  We expect that the results can be generalized to adiabatic methods in more general stochastic hybrid systems.
\keywords{adiabatic reduction \and piecewise deterministic Markov process \and stochastic bursting gene expression \and quasi-steady state assumption \and scaling limit}
 \subclass{92C45 \and 60Fxx \and 92C40 \and 60J25 \and 60J75}

\end{abstract}

\section{Introduction}

\label{intro}

The adiabatic reduction technique is often used to reduce the dimension of a dynamical system when known, or presumptive, fast and slow variables are present.
Adiabatic reduction results for deterministic systems of ordinary differential equations have been available since the work of \cite{Fenichel1979} and \cite{Tikhonov1952}. This technique has been extended to stochastically perturbed systems when the perturbation is a Gaussian distributed white noise, {\it cf.} \cite{Berglund2006}, \cite[Section 6.4]{Gardiner1985}, \cite[Chapter 4, Section 11.1]{Stratonovich:1963}, \cite{Titular:1978} and \cite{Wilemski:1976}. More recently, separation of time scales in discrete pure jump Markov processes were performed, using a master equation formalism \cite{Santillan2011} or a stochastic equation formalism \cite{Kang,Crudu2011}. These papers show that a fast stochastic process can be averaged in the slow time scale, or can induce kicks to the slow variable. However, to the best of our knowledge, this type of approximation has never been extended to the situation in which the (fast) perturbation is a jump Markov process in a piecewise deterministic Markov process.

Jump Markov processes are often used in modelling stochastic gene expressions with explicit bursting in either mRNA or proteins \cite{Friedman2006,Golding2005}, and have been employed as models for genetic networks \cite{Zeisler:2008} and in the context of excitable membranes \cite{Buckwar:2011,Pakdaman:2010,Riedler:2012}. Biologically, the `bursting' of mRNA or protein is simply a process in which there is a production of several molecules within a very short time. In the biological context of modelling stochastic gene expression, explicit models of bursting mRNA and/or protein production have been analyzed recently, either using a discrete  \cite{Shahrezaei2008} or a continuous formalism \cite{Friedman2006,Lei2009,Mackey2011} as even more experimental evidence from single-molecule visualization techniques has revealed the ubiquitous nature of this phenomenon  \cite{Elf2007,Golding2005,Ozbudak2002,Raj2009,Raj2006,Suter2011,Xie2008}.

Traditional models of gene expression are often composed of \textit{at least} two variables (mRNA and protein, and sometimes the promoter state). The use of a reduced one-dimensional model (protein concentration) has been justified so far by an argument concerning the stationary distribution \cite{Shahrezaei2008}.
However, it is clear that the two different models may have the same stationary distribution but very different dynamic behavior (for an example, see \cite{Mackey2011}).  The adiabatic reduction technique has been used in many studies ({\it cf.} \cite{Hasty:2000,Mackey2011}) to simplify the analysis of  stochastic gene expression dynamics, but without a rigorous mathematical justification.

The present paper gives a theoretical justification of the use of adiabatic reduction in a model of auto-regulation gene expression with a jump Markov process in mRNA transcription. We adopt a formalism based on density evolution (Fokker-Planck like) equations. Our results are of importance since they offer a rigorous justification for the use of adiabatic reduction to jump Markov processes. The model and mathematical results are presented in Sections \ref{sec:model}. Proof of the results are given in Section \ref{sec:proof}, with illustrative simulations in Section \ref{sec:ill}.

\section{Model and results}
\label{sec:model}

\subsection{Continuous-state bursting model}

A single round of expression consists of both mRNA transcription and the translation of proteins from mRNA. The mRNA transcription occurs in a burst like fashion depending on the promoter activity. In this study, we assume a simple feedback between the end product (protein) which binds  to its own promoter to regulate the transcription activity.

Let $X$ and $Y$ denote the concentrations of mRNA and protein respectively. A simple mathematical model of a single gene expression with feedback regulation and bursting in transcription is given by
\begin{eqnarray}
\label{eq:gene1}
\dfrac{dX}{dt}&=&-\gamma_1 X + \mathring{N}(h, \varphi(Y)),\\
\label{eq:gene2}
\dfrac{dY}{dt}&=&-\gamma_2 Y+\lambda_2 X.
\end{eqnarray}
Here $\gamma_1$ and $\gamma_2$ are degradation rates for mRNA and proteins respectively, $\lambda_2$ is the translational rate, and $\mathring{N}(h, \varphi(Y))$ describes the transcriptional burst that is assumed to be a compound Poisson \textit{white noise} occurring at a rate $\varphi$ with a non-negative jump size $\Delta X$ distributed with density $h$.

In the model equations \eqref{eq:gene1}-\eqref{eq:gene2}, the stochastic transcriptional burst is characterized by the two functions $\varphi$ and $h$. We always assume these two functions satisfy
\begin{eqnarray}
\label{eq:phi}
\varphi &\in& \mathcal{C}^{\infty}(\mathbb{R}^+,\mathbb{R}^+),\,\,\,\varphi \text{ and }\varphi' \text{ are bounded,}\ {\it i.e.}\ \underline{\varphi}\leq\varphi,\varphi'\leq\overline{\varphi}\\
\label{eq:h}
h&\in& \mathcal{C}^{\infty}(\mathbb{R}^+,\mathbb{R}^+)\,\,\,\text{ and }\int_{0}^\infty x^nh(x) d  x < \infty ,\quad \forall n\geq 1.
\end{eqnarray}
For a general density function $h$, the average burst size is given by
\begin{equation}
\label{eq:b}
b = \int_0^\infty x h(x)d x.
\end{equation}

\begin{remark}\label{rem:hill}
Hill functions are often used to model self-regulation in gene expression, so that $\varphi$ is given by
\begin{equation*}
 \varphi(y)=\varphi_0\frac{1+K y^n}{A + B y^n}
\end{equation*}
where $\varphi_0, A, B, K$ and $n$ are positive parameters  (see \cite{Mackey2011} for more details).

An exponential distribution of the burst jump size is often used in modelling gene expression, in agreement with experimental findings \cite{Xie2008}, so that the density function $h$ is given by
\begin{equation*}
h(\Delta X) = \dfrac{1}{b}e^{-\Delta X/b},
\end{equation*}
where $b$ is the average burst size.

The two functions $\varphi$ and $h$ here satisfy the assumptions \eqref{eq:phi}-\eqref{eq:h}.
\end{remark}

\subsection{Scalings}

The equations \eqref{eq:gene1}-\eqref{eq:gene2} are nonlinear, coupled, and analytically not easy to study. This paper provides an analytical understanding of the adiabatic reduction for \eqref{eq:gene1}-\eqref{eq:gene2} when mRNA degradation is a fast process, {\it i.e.}, $\gamma_1$ is ``large enough'' ($\gamma_1\gg \gamma_2$) but the average protein concentration remains normal. Rapid mRNA degradation has been observed in \textit{E. coli} (and other bacteria), in which mRNA is typically degraded within minutes, whereas most proteins have a lifetime longer than the cell cycle ($\geq 30$ minutes for \textit{E. coli}) \cite{Taniguchi:2010}.

In \eqref{eq:gene1}-\eqref{eq:gene2}, when $\gamma_1$ is large, other parameters have to be adjusted accordingly to maintain a normal level of protein. When there is no feedback regulation to the burst rate, the function $\varphi$ is independent of $Y$ (therefore $\varphi$ is a constant), and thus the average concentrations of mRNA and protein in a stationary state are
\begin{eqnarray}
\label{eq:Xeq}
X_{\mathrm{eq}} := \lim_{t\to\infty}\mathbb{E}[X(t)]&=& \frac{b\varphi}{\gamma_1},\\
\label{eq:Yeq}
Y_{\mathrm{eq}} :=\lim_{t\to\infty}\mathbb{E}[Y(t)]&=& \frac{\lambda_2}{\gamma_2}X_{\mathrm{eq}} = \frac{b\varphi\lambda_2}{\gamma_1\gamma_2}.
\end{eqnarray}
From \eqref{eq:Yeq}, when $\gamma_1$ is large enough ($\gamma_1 \gg \gamma_2$) and $Y_{\mathrm{eq}}$ remains at its normal level, one of the three quantities, $b$, $\varphi$, or $\lambda_2$ must be a large number as well. This observation holds even when there is a feedback regulation of  the burst rate. Thus, in general, we have three possible scalings (as $\gamma_1\to\infty$), each of which is biologically observed:
\textit{\begin{enumerate}
\item[\textnormal{(S1)}] Fast promoter activation/deactivation, so that the rate function $\varphi$ is a large number. In this case, if $\gamma_1\to\infty$, we assume the ratio $\varphi/\gamma_1$ is independent of $\gamma_1$.
\item[\textnormal{(S2)}] Fast transcription, so that the average burst size $b$ is a large number. From \eqref{eq:b}, this scaling indicates that the density function $h$ changes with the parameter $\gamma_1$ in a form $h(\Delta X) = \frac{1}{\gamma_1}h_0(\frac{\Delta X}{\gamma_1})$ with $h_0(\cdot)$ independent of $\gamma_1$.
\item[\textnormal{(S3)}] Fast translation, so that the translational rate $\lambda_2$ is a large number. In this case, if $\gamma_1\to\infty$, we assume the ratio $\lambda_2/\gamma_1$ is independent of $\gamma_1$.
\end{enumerate}}

These scalings are associated with different types of genes that display  different types of kinetics ({\it cf.} \cite{Schwa:2011,Suter2011}), and mathematically lead to different forms of the reduced dynamics. In this paper we determine the  effective reduced equations for equations \eqref{eq:gene1}-\eqref{eq:gene2} for each of the scaling conditions (S1)-(S3). Our main results are summarized below.

First, under the assumption (S1) (\textit{fast promoter activation/deactivation}), equations \eqref{eq:gene1}-\eqref{eq:gene2} can be approximated by a deterministic ordinary differential equation
\begin{equation}
\label{eq:odey0}
\dfrac{dY}{dt}=-\gamma_2 Y+\lambda_2 \psi(Y)
\end{equation}
where
\begin{equation}
\psi(Y) = b\varphi(Y)/\gamma_1.
\end{equation}

Next, under the scaling relations (S2)(\textit{fast transcription}) or (S3)(\textit{fast translation}), equations \eqref{eq:gene1}-\eqref{eq:gene2} are reduced to a single stochastic differential equation
\begin{equation}
\label{eq:gene2_reduce}
\dfrac{dY}{dt}=-\gamma_2 Y + \mathring{N}(\bar{h}(\Delta Y),\varphi(Y))
\end{equation}
containing a jump Markov process, and the density $\bar{h}$ for the newly defined process is given by $h$ through
\begin{equation}
\label{eq:hbar3}
\bar{h}(\Delta Y) = \left (\frac{\lambda_2}{\gamma_1}\right )^{-1}h\left ( \left (\frac{\lambda_2}{\gamma_1}\right )^{-1}\Delta Y\right ).
\end{equation}
In particular, with the scaling (S2), we have
\begin{equation}
\label{eq:hbar2}
\bar{h}(\Delta Y) = \frac{1}{\lambda_2}h_0\left (\frac{\Delta Y}{\lambda_2}\right ).
\end{equation}

These results can be understood with the following simple arguments. When $\gamma_1 \to \infty$, applying  a standard quasi-equilibrium assumption we have
$$\frac{dX}{dt}\approx 0,$$
which yields
\begin{equation}
\label{eq:formal}
 X(t) \approx \frac{1}{\gamma_1}\mathring{N}(h,\varphi(Y)).
\end{equation}
In the case of the scaling (S1), the jumps occur with high frequency and an average burst size $b$. Thus, $X(t)$ approaches the statistical average ($X(t)\approx b \varphi(Y)/\gamma_1$) for a given value $Y$, which gives  \eqref{eq:odey0}.
Under scalings (S2) or (S3), substituting \eqref{eq:formal} into \eqref{eq:gene2} yields
\begin{eqnarray*}
 \frac{dY}{dt} &\approx& -\gamma_2 Y+ \frac{\lambda_2}{\gamma_1}\mathring{N}(h,\varphi(Y))  \\
&\approx& -\gamma_2 Y +\mathring{N}\left (\bar{h}, \varphi(Y)\right).
\end{eqnarray*}
Exact statements for the results and their mathematical proofs are given below.

\subsection{Density evolution equations and main results}
The main results are based on the density evolution equations, and show that the evolution equations obtained from equations \eqref{eq:gene1}-\eqref{eq:gene2} and those from \eqref{eq:odey0} or \eqref{eq:gene2_reduce} are consistent with each other when $\gamma_1 \to +\infty$ under the appropriate scaling. The existence of densities for such processes has been studied in \cite{Mackey2008,Tyran-Kaminska2009}.

Let $u(t,x,y)$ be the density function of $(X(t),Y(t))$ at time $t$ obtained from the solutions of equation \eqref{eq:gene1}-\eqref{eq:gene2}. The evolution of the density $u(t,x,y)$ is governed by ({\it cf.} \cite{Mackey2008})
\begin{equation}
\label{eq:den2}
\begin{array}{rcl}
\displaystyle \dfrac{\partial u(t,x,y)}{\partial t} &=& \displaystyle\dfrac{\partial\ }{\partial x}[\gamma_1 x u(t,x,y)] - \dfrac{\partial\ }{\partial y}[(\lambda_2 x - \gamma_2 y)u(t,x,y)]\\
\displaystyle &&{} + \displaystyle\int_0^x \varphi(y) u(t,z,y) h(x-z)dz - \varphi(y)u(t,x,y)
\end{array}
\end{equation}
when $(t,x,y)\in \mathbb{R}^+\times\mathbb{R}^+\times\mathbb{R}^+$. The corresponding density function of $Y(t)$ is given by
\begin{equation}
u_0(t,y)=\int_0^\infty u(t,x,y) d x.
\end{equation}

In this paper, we prove that when $\gamma_1\to\infty$ the density function $u_0(t,y)$ approaches the density $v(t,y)$ for solutions of either the deterministic equation \eqref{eq:odey0} or the stochastic differential equation \eqref{eq:gene2_reduce} depending on the scaling. Evolution of the density function for equation \eqref{eq:odey0} is given by \cite{Lasota1985}
\begin{equation}
\label{eq:den0}
\dfrac{\partial v(t,y)}{\partial t} = -\dfrac{\partial\ }{\partial y}[-\gamma_2 y v(t,y) + \lambda_2 \psi(y) v(t,y)],
\end{equation}
where
\begin{equation}
\psi(y) = b \varphi(y)/\gamma_1.
\end{equation}
Evolution of the density function for equation \eqref{eq:gene2_reduce} is given by
\begin{equation}
\label{eq:den1}
\dfrac{\partial v(t,y)}{\partial t} = \dfrac{\partial\ }{\partial y}[ \gamma_2 y v(t,y)] + \int_0^y \varphi(z) v(t,z) \bar{h}(y-z) d z - \varphi(y) v(t,y).
\end{equation}
Here $\bar{h}$ is related to $h$ through
\begin{equation}
\label{eq:hS3}
\bar{h}(y) = \frac{\gamma_1}{\lambda_2} h \left (\frac{\gamma_1}{\lambda_2} y \right ).
\end{equation}

We note that when $\varphi$ and $h$ satisfy \eqref{eq:phi}-\eqref{eq:h}, existence of the above densities has been proved in \cite{Mackey2008} and \cite{Tyran-Kaminska2009}. In particular, for a given initial density function
\begin{equation}
\label{eq:ini2}
u(0,x,y) = p(x,y),\quad  0< x,y < +\infty
\end{equation}
that satisfies
\begin{equation}
\label{eq:con1}
p(x,y)\geq 0,\quad \int_0^\infty \int_0^\infty p(x,y) d x d y = 1,
\end{equation}
there is a unique solution $u(t,x,y)$   of \eqref{eq:den2} that satisfies the initial condition \eqref{eq:ini2} and
\begin{equation}
u(t,x,y)\geq 0,\quad \int_0^\infty \int_0^\infty u(t,x,y) d x d y = 1
\end{equation}
for all $t\in\mathbb{R}^+$.

We can rewrite the equations \eqref{eq:den0} and \eqref{eq:den1} in the form
\begin{equation}
\label{eq:wden}
\dfrac{\partial v(t,y)}{\partial t} = \mathcal{T} v(t,y),
\end{equation}
where $\mathcal{T}$ is a linear operator defined by the right hand side of \eqref{eq:den0} or \eqref{eq:den1}.
\begin{definition}
\label{def1}
A smooth function $f:\mathbb{R}^+\to\mathbb{R}^+$ is a \textit{test function} if $f(y)$ has compact support and $f^{(k)}(0) = 0$ for any $k=0,1,2,\cdots$.
An integrable function $v(t,y): \mathbb{R}^+\times \mathbb{R}^+\mapsto \mathbb{R}^+$ is said to be a \textit{weak solution} of \eqref{eq:wden} if for any test function $f(y)$,
\begin{equation}
\int_{0}^\infty \left(\dfrac{\partial v(t,y)}{\partial y}  - \mathcal{T}v(t,y)\right)f (y) d y = 0,\quad \forall t>0.
\end{equation}
\end{definition}

\begin{remark}
It is obvious that any classical solution of \eqref{eq:wden} is also a weak solution.
\end{remark}

The main result of this section, given below, shows that when $\gamma_1$ is large enough, the marginal density of $Y(t)$, $u_0(t,y;\gamma_1)$, as defined below in \eqref{eq:u0}, gives an approximation of a weak solution of \eqref{eq:den0} or \eqref{eq:den1}.

\begin{theorem}
\label{th:m3}
Let $u(0,x,y) = p(x,y)\in\mathcal{C}^{\infty}(\mathbb{R^+}^2)$ and assume that $p(x,y)$  satisfies
\begin{equation}
\label{eq:con2}
\int_0^\infty x^n p(x,y)dx < +\infty,\quad  y > 0,\ n=0,1,2,\cdots.
\end{equation}
For any $\gamma_1 > 0$, let $u(t,x,y;\gamma_1)$ be the associated solution of \eqref{eq:den2}, and define
\begin{equation}
\label{eq:u0}
u_0(t,y;\gamma_1) = \int_0^\infty u(t,x,y;\gamma_1) d x.
\end{equation}
Similarly,
\begin{equation*}
p_0(y) = \int_0^\infty p(x,y) d x.
\end{equation*}
\begin{enumerate}
\item[\textnormal{(1)}] Under the scaling (S1), when $\gamma_1\to\infty$, $u_0(t,y;\gamma_1)$ approaches a weak solution of \eqref{eq:den0} $v(t,y)$ with initial condition $v(0,y)=p_0(y)$.
\item[\textnormal{(2)}] Under the scaling (S2) or (S3), when $\gamma_1\to\infty$, $u_0(t,y;\gamma_1)$ approaches a weak solution of \eqref{eq:den1} $v(t,y)$ with initial condition $v(0,y)=p_0(y)$.
\end{enumerate}
\end{theorem}

From Definition \ref{def1}, Theorem \ref{th:m3} means that for any test function $f(y)$,
\begin{equation}
\label{eq:lsd}
\lim_{\gamma_1\to\infty}\int_0^\infty \left(\dfrac{\partial u_0(t,y;\gamma_1)}{\partial t} - \mathcal{T} u_0(t,y;\gamma_1)\right) f(y) d y = 0,\quad \forall t> 0.
\end{equation}
In the next section, we prove \eqref{eq:lsd} for the three scalings respectively.

\section{Proof of the main results}
\label{sec:proof}

Before proving Theorem \ref{th:m3}, we first examine the marginal moments under different scalings.

\subsection{Scaling of the marginal moment}\label{section:moment}

\begin{proposition}\label{prop:moment}
Let $(X(t),Y(t))$ be the solutions of \eqref{eq:gene1}-\eqref{eq:gene2}, $\mu_k(t)=\mathbb{E}\big{[}X(t)^k\big{]}$ and $\nu_k(t)=\mathbb{E}\big{[}Y(t)X(t)^k\big{]}$. Suppose $\mu_k(0)<\infty$ and $\nu_k(0)<\infty$, then $\mu_k(t)<\infty$ and $\nu_k(t)<\infty$ for all $t$. Moreover, for any fixed $t>0$:
\begin{enumerate}
\item If the scaling (S1) holds, both $\mu_k(t)$ and $\nu_k(t)$ are uniformly bounded above and below when $\gamma_1$ is large enough.
\item If the scaling (S2) holds, when $\gamma_1$ is large enough, for $k\geq 1$,
\begin{equation}
\mu_k(t)\sim \gamma_1^{k-1},\quad \nu_k(t)\sim\gamma_1^{k-1},
\end{equation}
and $\nu_0(t)$ is uniformly bounded above and below.
\item If the scaling (S3) holds, when $\gamma_1$ is large enough, for $k\geq 1$,
 \begin{equation}
\mu_k(t)\sim \gamma_1^{-1},\quad \nu_k(t)\sim\gamma_1^{-1},
\end{equation}
and $\nu_0(t)$ is uniformly bounded above and below.
\end{enumerate}
\end{proposition}

\begin{proof}
For the two-dimensional stochastic differential equation \eqref{eq:gene1}-\eqref{eq:gene2}, the associated infinitesimal generator $\mathcal{A}$ is defined as \cite[Theorem 5.5]{Davis1984}
\begin{eqnarray}
\label{generator2d}
\mathcal{A}g(x,y) &=& -\gamma_1 x\dfrac{\partial g}{\partial x}+(\lambda_2 x-\gamma_2 y)\dfrac{\partial g}{\partial y}  \\
&&{} + \varphi(y)\Bigg{(}\int_x^{\infty}h(z-x)g(z,y)dz-g(x,y)\Bigg{)} \nonumber
\end{eqnarray}
for any $g\in \mathcal{C}^1({\mathbb{R}^+}\times \mathbb{R}^+)$.
The operator $\mathcal{A}$ is the adjoint of the operator acting on the right hand side of the evolution equation of the density \eqref{eq:den2}.
Moreover, for any $g\in \mathcal{C}^1(\mathbb{R}^+\times\mathbb{R}^+)$, we have
\begin{equation}
\label{eq:weak}
\frac{d\ }{dt}\mathbb{E}g(X_t,Y_t)=\mathbb{E}\mathcal{A}(g(X_t,Y_t)),
\end{equation}
provided both terms on the right hand side of \eqref{generator2d} are finite. The proposition is proved through calculations of \eqref{eq:weak}.

To obtain estimations for $\mu_k$, a straightforward calculation from  \eqref{generator2d} yields
\begin{eqnarray*}
 \mathcal{A}\,x^k &=& -\gamma_1 k x^k+\varphi(y)\Bigg{(}\int_x^{\infty}h(z-x)(z-x+x)^kdz-x^k\Bigg{)}\\
&=& -\gamma_1 kx^k+\varphi(y)\sum_{i=0}^{k-1}{k\choose i}x^i\int_x^{\infty}h(z-x)(z-x)^{k-i}dz\\
&=& -\gamma_1 kx^k+\varphi(y)\sum_{i=0}^{k-1}{k\choose i}x^i\mathbb{E}^{k-i} h,\\
\end{eqnarray*}
where
$$\mathbb{E}^jh = \int_0^\infty x^j h(x) d x.$$
Thus, \eqref{eq:weak} yields
\begin{equation}
\dfrac{d \mu_k(t)}{d t} = - \gamma_1 k \mu_k(t) +  \sum_{i=0}^{k-1}{k \choose i}\mathbb{E} \left[\varphi(Y_t) X(t)^i\right] \mathbb{E}^{k-1} h.
\end{equation}
We then obtain, with the assumption \eqref{eq:phi},
\begin{equation}
\label{ineq:moment}
\underline{\varphi}\sum_{i=0}^{k-1}{k\choose i}\mu_i(t)\mathbb{E}^{k-i} h \leq \dot{\mu_k}(t) + \gamma_1k\mu_k(t) \leq \overline{\varphi}\sum_{i=0}^{k-1}{k\choose i}\mu_i(t)\mathbb{E}^{k-i} h.
\end{equation}

Now, we can obtain estimations of $\mu_k$ for different scalings from \eqref{ineq:moment}

1. Assume the scaling (S1) so that both $\overline{\varphi}/\gamma_1$ and $\underline{\varphi}/\gamma_1$ are independent of $\gamma_1$ when $\gamma_1$ is large enough. Applying Gronwall's inequality to equation \eqref{ineq:moment} with $k=1$ yields, for all $t>0$,
$$
\frac{\underline{\varphi}\,b}{\gamma_1} + \left[\mu_1(0) - \frac{\underline{\varphi}\,b}{\gamma_1}\right] e^{-\gamma_1 t}\leq\mu_1(t)\leq\frac{\overline{\varphi}\,b}{\gamma_1} + \left[\mu_1(0) - \frac{\overline{\varphi}\,b}{\gamma_1}\right] e^{-\gamma_1 t}.
$$
Thus, $\mu_1(t)$ is uniformly bounded above and below when $\gamma_1$ is large enough.

Iteratively, for all $t>0$ and $k>1$, there are constants $\bar{c}_k, \underline{c}_k>0$ independent of $\gamma_1$ such that
$$
\frac{\underline{\varphi}\,\underline{c}_k}{k\gamma_1} + \left[\mu_k(0) - \frac{\underline{\varphi}\,\underline{c}_k}{k \gamma_1}\right]e^{- k\gamma_1 t}\leq\mu_k(t)\leq \frac{\overline{\varphi}\,\overline{c}_k}{k\gamma_1} + \left[\mu_k(0) - \frac{\overline{\varphi}\,\overline{c}_k}{k \gamma_1}\right]e^{- k\gamma_1 t},
$$
and hence $\mu_k(t)$ is uniformly bounded above and below when $\gamma_1$ is large enough.

2. Assume the scaling (S2) so that $\mathbb{E}^{k-i}h\sim \gamma_1^{k-i}$ when $\gamma_1$ is large enough. We note $\mu_0(t) = 1$, and therefore inductively, for any $t$ and $k\geq 1$,
$$
\frac{\underline{\varphi}\,\mathbb{E}^kh}{k\gamma_1} + O(\gamma_1^{k-2})\leq\mu_k(t)\leq\frac{\overline{\varphi}\,\mathbb{E}^kh}{k\gamma_1} + O(\gamma_1^{k-2}).
$$
Thus, we have $\mu_k(t)\sim\gamma_1^{k-1}$ when $\gamma_1$ is large enough.

3. Assume the scaling (S3) so that $\lambda_2/\gamma_1$ is independent of $\gamma_1$ when $\gamma_1$ is large enough. Calculations similar to those in case (S1) gives $\mu_k(t)\sim \gamma_1^{-1}$.

Analogous results for $\nu_k(t)$ are obtained with similar calculations with \linebreak $g(x,y) = x^k y$ in \eqref{generator2d}. Namely, we have
\begin{eqnarray*}
 \mathcal{A}\,x^ky &=& -(\gamma_1 k+\gamma_2) x^ky+\lambda_2 x^{k+1}+\varphi(y)\sum_{i=0}^{k-1}{k\choose i}x^iy\mathbb{E}^{k-i} h.
\end{eqnarray*}
Thus, when $k=0$, we have
\begin{equation*}
 \dot{\nu_0}=-\gamma_2 \nu_0+\lambda_2 \mu_1,
\end{equation*}
and for $k\geq 1$,
\begin{eqnarray*}
 &&\quad-(\gamma_1k+\gamma_2)\nu_k(t)+\lambda_2\mu_{k+1}+\underline{\varphi}\sum_{i=0}^{k-1}{k\choose i}\nu_i(t)\mathbb{E}^{k-i} h\\
 &&\leq\ \dot{\nu_k}(t)\ \leq\ -(\gamma_1k+\gamma_2)\nu_k(t)+\lambda_2\mu_{k+1}+\overline{\varphi}\sum_{i=0}^{k-1}{k\choose i}\nu_i(t)\mathbb{E}^{k-i} h.
\end{eqnarray*}
Then $\nu_0$ is uniformly bounded for each scaling (S1), (S2), and (S3). Then, iteratively using the inequalities for $\dot{\nu_k}$, the scaling of $\mu_{k+1}$ and Gronwall's inequality yields the desired result for each scaling.
\end{proof}

\begin{remark}
\label{rem3}
Define the marginal moments
\begin{equation}
u_k(t,y) = \int_0^\infty x^k u(t,x,y) d x,
\end{equation}
then
$$\mu_k(t) = \int_0^\infty u_k(t,y) d y.$$
Hence the integrals $\int_0^\infty u_k(t,y)d y$ satisfy the same scaling as $\mu_k(t)$ when $\gamma_1\to\infty$.
\end{remark}

\begin{remark}
\label{rem4}
From \eqref{ineq:moment}, when $\gamma_1\to\infty$ the moments $\dot{\mu}_k(t)$ have the same scaling as $\mu_k(t)$. Moreover, the same scalings are valid for the integrals $\displaystyle\int_0^\infty \frac{\partial u_k(t,y)}{\partial t} d y$.
\end{remark}

\subsection{Proof of Theorem \ref{th:m3}}
\begin{proof} Throughout the proof, we omit $\gamma_1$ in the solution $u(t,x,y;\gamma_1)$ and in the marginal density $u_0(t,y;\gamma_1)$, and keep in mind that they are dependent on the parameter $\gamma_1$ through equation \eqref{eq:den2}.

First, from Section \ref{section:moment} and \eqref{eq:con2}, the marginal moments
\begin{equation}
u_n(t,y)=\int_0^\infty x^n u(t,x,y) d x,
\end{equation}
are well defined for $t>0$, $y>0$ and $n\geq 0$.  Hence
\begin{equation}
\label{eq:xx}
\begin{aligned}
\lim_{x\to\infty} x^n u(t,x,y) &= 0,\quad \forall t, y, n>0. \\
\lim_{x\to 0} x^n u(t,x,y) &= 0,\quad \forall t, y, n\geq 1.
\end{aligned}
\end{equation}
From \eqref{eq:den2}, we multiply by $x^n$ and integrate on both sides. By \eqref{eq:xx}, we have
\begin{equation}
\begin{array}{rcl}
\displaystyle\frac{\partial u_n}{\partial t} &=& \displaystyle- n \gamma_1 u_n - \lambda_2 \frac{\partial u_{n+1} }{\partial y}  + \gamma_2 \frac{\partial (y u_n) }{\partial y} \\
&&\displaystyle{} + \int_0^\infty \int_0^x \varphi(y) x^n u(t, z, y) h(x-z) d z dx - \varphi(y) u_n.
\end{array}
\end{equation}
Since
$$\int_0^\infty \int_0^x \varphi(y) x^n u(t,z,y) h(x-z) d z dx= \sum_{j=0}^n {n \choose j} \varphi(y) u_{n-j} \mathbb{E}^j h,$$
we have
\begin{equation}
\frac{\partial u_n}{\partial t} = - n \gamma_1 u_n - \lambda_2 \frac{\partial u_{n+1} }{\partial y}  + \gamma_2 \frac{\partial (y u_n) }{\partial y}   + \varphi(y)\sum_{j=1}^n {n \choose j} u_{n-j} \mathbb{E}^j h.
\end{equation}

In particular, when $n=0$,
\begin{equation}
\label{eq:tt0}
\frac{\partial u_0}{\partial t} =  - \lambda_2 \frac{\partial u_{1} }{\partial y}  + \gamma_2 \frac{\partial (y u_0) }{\partial y},
\end{equation}
and when $n\geq 1$,
\begin{equation}
\label{eq:tt1}
\frac{1}{\gamma_1}\frac{\partial u_n}{\partial t} = - n  u_n - \frac{\lambda_2}{\gamma_1} \frac{\partial u_{n+1} }{\partial y}  + \frac{\gamma_2}{\gamma_1} \frac{\partial (y u_n) }{\partial y}   + \frac{1}{\gamma_1}\varphi(y)\sum_{j=1}^n {n \choose j} u_{n-j} \mathbb{E}^j h.
\end{equation}
Thus, for any $n\geq 1$,
\begin{eqnarray}
\label{eq:tt2}
u_n &=& - \frac{\lambda_2}{n\gamma_1} \frac{\partial u_{n+1} }{\partial y}  + \frac{\gamma_2}{n\gamma_1} \frac{\partial (y u_n) }{\partial y}\\
&&{}  + \frac{1}{n\gamma_1}\varphi(y)\sum_{j=1}^n {n \choose j} u_{n-j} \mathbb{E}^j h -\frac{1}{n\gamma_1}\dfrac{\partial u_n}{\partial t}.\nonumber
\end{eqnarray}

Now, we are ready to prove the results for the three scalings by iteratively  calculating $u_1$ from \eqref{eq:tt2}.

For the scaling (S1) so $\varphi(y) \sim \gamma_1$, and (here $b = \mathbb{E}h$)
\begin{equation}
\label{eq:tt3}
u_1 = \frac{b \varphi(y)}{\gamma_1} u_0 + \frac{1}{\gamma_1}\left[\frac{\partial\ }{\partial y}(\gamma_2 y u_1 - \lambda_2 u_2) - \frac{\partial u_1}{\partial t}\right].
\end{equation}
Substituting \eqref{eq:tt3} into \eqref{eq:tt0},  we obtain
\begin{equation}
\frac{\partial u_0}{\partial t} = \frac{\partial\ }{\partial y}[\gamma_2 y u_0 - \lambda_2 \psi(y) u_0] - \frac{\lambda_2}{\gamma_1} \frac{\partial\ }{\partial y}\left[\frac{\partial\ }{\partial y}(\gamma_2 y u_1 - \lambda_2 u_2) - \frac{\partial u_1}{\partial t}\right],
\end{equation}
where $\psi(y) = b \varphi(y)/\gamma_1$. Now, we only need to show that for any test function $f(y)$,
\begin{equation}
\label{eq:f0}
\lim_{\gamma_1\to\infty} \dfrac{\lambda_2}{\gamma_1} \int_0^\infty f(y) \frac{\partial\ }{\partial y}\left[\frac{\partial\ }{\partial y}(\gamma_2 y u_1 - \lambda_2 u_2) - \frac{\partial u_1}{\partial t}\right] d y = 0,\quad \forall t > 0.
\end{equation}
We note that the integral
\begin{eqnarray*}
\int_0^\infty f(y) \frac{\partial\ }{\partial y}\left[\frac{\partial\ }{\partial y}(\gamma_2 y u_1 - \lambda_2 u_2) - \frac{\partial u_1}{\partial t}\right] d y &=& - \int_0^\infty f'(y) \frac{\partial u_1}{\partial t} d y\\
&&{} + \int_0^\infty f''(y) (\gamma_2 y u_1 - \lambda_2 u_2) d y
\end{eqnarray*}
is uniformly bounded when $\gamma_1$ is large enough, \eqref{eq:f0} is straightforward from the Remarks \ref{rem3} and \ref{rem4}. Thus, we conclude that $u_0(t,y)$ approaches a weak solution of \eqref{eq:den0} and (1) of Theorem \ref{th:m3} is proved.

For the scaling (S2) so that $\mathbb{E}^j h \sim \gamma_1^j$ when $\gamma_1\to\infty$, let
\begin{equation}
\label{eq:bj}
b_j = \gamma_1^{-j}\mathbb{E}^j h,\quad (j=0,1,\cdots)
\end{equation}
which are independent of $\gamma_1$ when $\gamma_1\to\infty$. Hence, from \eqref{eq:tt2} and Proposition \ref{prop:moment}, we have
\begin{eqnarray*}
\gamma_1^{-(n-1)}u_n &=& - \frac{\lambda_2}{n}\dfrac{\partial (\gamma_1^{-n} u_{n+1})}{\partial y} + \frac{\gamma_2}{n\gamma_1}\frac{\partial (y \gamma_1^{-(n-1)}u_n)}{\partial y} + \frac{1}{n}\varphi(y) u_0 b_n\\
&&{} + \frac{1}{n\gamma_1}\varphi(y)\sum_{j=1}^{n-1}{n\choose j} \gamma_1^{-(n-j-1)}u_{n-j}b_j - \frac{1}{n \gamma_1} \frac{\partial (\gamma_1^{-(n-1)}u_n)}{\partial t}\\
&=& \frac{1}{n}b_n \varphi(y) u_0 - \frac{\lambda_2}{n}\dfrac{\partial (\gamma_1^{-n} u_{n+1})}{\partial y} + \frac{1}{n\gamma_1}C_n(t,y),
\end{eqnarray*}
where
$$C_n(t,y) = \gamma_2 \frac{\partial (y \gamma_1^{-(n-1)}u_n)}{\partial y} + \varphi(y)\sum_{j=1}^{n-1}{n\choose j}\gamma_1^{-(n-j-1)}u_{n-j}b_j - \frac{\partial (\gamma_1^{-(n-1)} u_n)}{\partial t}.$$
Therefore,
\begin{eqnarray*}
u_1 &=& b_1 \varphi(y)u_0 - \lambda_2 \frac{\partial\ }{\partial y} [\gamma_1^{-1} u_2] + \frac{1}{\gamma_1}C_1(t,y)\\
&=&b_1 \varphi(y)u_0 - \lambda_2 \frac{\partial\ }{\partial y} \left[\frac{1}{2}b_2 \varphi(y) u_0 - \frac{\lambda_2}{2}\frac{\partial (\gamma_1^{-2}u_3)}{\partial y} + \frac{1}{2\gamma_1}C_2(t,y)\right] + \frac{1}{\gamma_1}C_1(t,y)\\
&=& b_1 \varphi(y)u_0 - b_2 \frac{\lambda_2}{2!} \frac{\partial}{\partial y}(\varphi(y)u_0)  + \frac{\lambda_2^2}{2!}\frac{\partial^2}{\partial y^2}\left[\frac{1}{3}b_3 \varphi(y)u_0 - \frac{\lambda_2}{3}\frac{\partial (\gamma_1^{-3}u_4)}{\partial y}+\frac{1}{3\gamma_1}C_3(t,y)\right]\\
&&{} + \frac{1}{\gamma_1}C_1(t,y) - \frac{\lambda_2}{2! \gamma_1} \frac{\partial\ }{\partial y}C_2(t,y)\\
&&\cdots\cdots\cdots\cdots\\
&=&\sum_{k=1}^\infty \frac{(-\lambda_2)^{k-1}}{k!} b_k\frac{\partial^{k-1}\ }{\partial y^{k-1}}(\varphi(y) u_0) + \frac{1}{\gamma_1}\sum_{k=1}^{\infty}\frac{(-\lambda_2)^{k-1}}{k!} \frac{\partial^{k-1}\ }{\partial y^{k-1}} C_k(t,y).
\end{eqnarray*}
Thus, denote
$$C(t,y) = -\lambda_2 \frac{\partial\ }{\partial y}\left[\sum_{k=1}^{\infty}\frac{(-\lambda_2)^{k-1}}{k!} \frac{\partial^{k-1}\ }{\partial y^{k-1}} C_k(t,y)\right] =  \sum_{k=1}^\infty\frac{(-\lambda_2)^k}{k!}\frac{\partial^k\ }{\partial y^k}C_k(t,y)$$
and from \eqref{eq:bj}, we have
\begin{eqnarray}
-\lambda_2 \frac{\partial u_1}{\partial y}&=& \sum_{k=1}^\infty \frac{(-\lambda_2)^{k}}{k!} (\gamma_1^{-k} \mathbb{E}^{k}h) \frac{\partial^{k}\ }{\partial y^{k}} (\varphi(y)u_0) + \frac{1}{\gamma_1}C(t,y)\nonumber\\
&=& \sum_{k=1}^\infty \frac{1}{k!} \left (-\frac{\lambda_2}{\gamma_1}\right )^k \left (\int_0^\infty x^{k} h(x) dx\right ) \frac{\partial^k}{\partial y^k}(\varphi(y) u_0)+ \frac{1}{\gamma_1}C(t,y)\nonumber\\
&=&\int_0^\infty \bar{h}(x) \left[\sum_{k=1}^\infty \frac{1}{k!} (-x)^{k} \frac{\partial^k}{\partial y^k}(\varphi(y) u_0)\right] dx+ \frac{1}{\gamma_1}C(t,y)\nonumber\\
&=&\int_0^\infty \bar{h}(x) (\varphi(y-x)u_0(t,y-x) - \varphi(y)u_0(t,y)) d x+ \frac{1}{\gamma_1}C(t,y)\nonumber\\
&=&\int_0^\infty \bar{h}(x)\varphi(y-x)u_0(t,y-x) d x - \varphi(y) u_0(t,y)+ \frac{1}{\gamma_1}C(t,y)\nonumber\\
&=&-\int_y^{-\infty} \bar{h}(y-z)\varphi(z) u_0(t,z) d z - \varphi(y) u_0(t,y)+ \frac{1}{\gamma_1}C(t,y)\nonumber\\
\label{eq:u1s2}
&=&\int_0^y \bar{h}(y-z) \varphi(z) u_0(t,z) d z - \varphi(y)u_0(t,y) + \frac{1}{\gamma_1}C(t,y).
\end{eqnarray}
Here we note $\varphi(z) = 0$ when $z<0$.

For any test function $f(y)$, similar to the argument in the scaling (S1), the integral
$$\int_0^\infty C(t,y) f(y) d y$$
is uniformly bounded when $\gamma_1$ is large enough, and hence
$$\lim_{\gamma_1\to\infty}\frac{1}{\gamma_1}\int_0^\infty C(t,y)f (y) d y = 0,\forall t>0.$$
Therefore, from \eqref{eq:tt0} and \eqref{eq:u1s2}, when $\gamma_1\to \infty$, $u_0$ approaches a weak solution of \eqref{eq:den1}, and (2) in Theorem \ref{th:m3} is proved.

Now, we consider the scaling (S3) so $\lambda_2/\gamma_1$ is independent of $\gamma_1$ when $\gamma_1\to\infty$. From \eqref{eq:tt2} and Proposition \ref{prop:moment}, we have
\begin{eqnarray*}
u_n &=&-\frac{1}{n}\frac{\lambda_2}{\gamma_1} \frac{\partial u_{n+1}}{\partial y} + \frac{\gamma_2}{n \gamma_1} \frac{\partial (y u_n)}{\partial y} + \frac{1}{n\gamma_1} \varphi(y) u_0 \mathbb{E}^n h   \\
&& + \frac{1}{n \gamma_1} \varphi(y) \sum_{j=1}^{n-1} {n \choose j} u_{n-j} \mathbb{E}^j h - \frac{1}{n \gamma_1} \frac{\partial u_n}{\partial t} \\
&=& \frac{1}{n \gamma_1} \varphi(y) u_0 \mathbb{E}^n h - \frac{1}{n} \frac{\lambda_2}{\gamma_1} \frac{\partial u_{n+1}}{\partial y} + \frac{1}{n\gamma_1}R_n(t,y),
\end{eqnarray*}
where
$$R_n(t,y) = \gamma_2 \frac{\partial (y u_n)}{\partial y} + \varphi(y) \sum_{j=1}^{n-1}{n \choose j} u_{n-j} \mathbb{E}^j h - \frac{\partial u_n}{\partial t}.$$
Therefore,
\begin{eqnarray*}
u_1&=& \frac{1}{\gamma_1}\varphi(y) u_0 \mathbb{E}^1 h -  \frac{\lambda_2}{\gamma_1}\frac{\partial}{\partial y} u_2 + \frac{1}{\gamma_1}R_1(t,y)\\
&=& \frac{1}{\gamma_1}\varphi(y) u_0 \mathbb{E}^1 h - \frac{\lambda_2}{\gamma_1} \frac{\partial}{\partial y}\left[\frac{1}{2\gamma_1}\varphi(y)u_0\mathbb{E}^2 h - \frac{1}{2}\frac{\lambda_2}{\gamma_1}\frac{\partial}{\partial y} u_3 + \frac{1}{2\gamma_1}R_2(t,y)\right]\\
&&{} + \frac{1}{\gamma_1}R_1(t,y)\\
&=& \frac{1}{\gamma_1} \varphi(y)u_0\mathbb{E}^1 h - \frac{1}{2!} \frac{\lambda_2}{\gamma_1^2}\mathbb{E}^2 h \frac{\partial}{\partial y}[\varphi(y) u_0]\\
&&{} + \frac{1}{2!} (\frac{\lambda_2}{\gamma_1})^2 \frac{\partial}{\partial y}\left[\frac{1}{3\gamma_1}\varphi(y) u_0 \mathbb{E}^3 h - \frac{1}{3}\frac{\lambda_2}{\gamma_1}\frac{\partial}{\partial u_4}\right]\\
&&{} + \frac{1}{\gamma_1}\sum_{k=1}^3\frac{1}{k!}(-\frac{\lambda_2}{\gamma_1})^{k-1} \frac{\partial^{k-1}}{\partial y^{k-1}}R_k(t,y)\\
&&\cdots\cdots\cdots\\
&=&-\frac{1}{\lambda_2}\sum_{k=1}^\infty \frac{1}{k!} (-\frac{\lambda_2}{\gamma_1})^k \mathbb{E}^k h \frac{\partial^{k-1}}{\partial y^{k-1}}[\varphi(y)u_0]\\
&&{} + \frac{1}{\gamma_1}\sum_{k=1}^\infty\frac{1}{k!}(-\frac{\lambda_2}{\gamma_1})^{k-1} \frac{\partial^{k-1}}{\partial y^{k-1}}R_k(t,y).
\end{eqnarray*}
Denote
$$R(t,y) = \sum_{k=1}^\infty \frac{1}{k!} (-\frac{\lambda_2}{\gamma_1})^k \frac{\partial^k\ }{\partial y^k}R_k(t,y),$$
and in a manner similar to the above argument, we have
\begin{eqnarray}
-\lambda_2 \frac{\partial u_1}{\partial y}&=& \sum_{k=1}^\infty\frac{1}{k!}(-\frac{\lambda_2}{\gamma_1})^k \mathbb{E}^k h \frac{\partial^k}{\partial y^k} [\varphi(y) u_0] + R(t,y)\nonumber\\
\label{eq:u1s3}
&=&\int_0^y \bar{h}(y-z)\varphi(z)u_0(t,z)dz - \varphi(y)u_0(t,y) + R(t,y).
\end{eqnarray}
Finally, we note $\mu_k(t)\sim \gamma_1^{-1}$ in the scaling (S3), hence for any test function $f(y)$,
$$\lim_{\gamma_1\to\infty}\int_0^\infty R(t,y) f(y) = 0.$$
Thus, from \eqref{eq:tt0} and \eqref{eq:u1s3}, when $\gamma_1\to \infty$, $u_0$ approaches to a weak solution of \eqref{eq:den1}, and (3) in Theorem \ref{th:m3} is proved.
\end{proof}

\section{Illustration}
\label{sec:ill}

We performed numerical simulations on \eqref{eq:gene1}-\eqref{eq:gene2} to illustrate the results in previous sections. In our simulations, we took parameter values so that $\gamma_1$ increases with the scaling (S2). As the intensity of the jumps is bounded, we used an accept/reject numerical scheme to simulate jump times, and used the exact solution of \eqref{eq:gene1}-\eqref{eq:gene2} between the jumps (the equations are linear between jumps).
For a given set of parameters, we simulate a trajectory for a sufficiently long time (a bound on the convergence rate can be obtained by the coupling method, see \cite{Bardet2011}) so that the stochastic process reaches its stationary state. We then computed its equilibrium density (as well as the first and second moments) by sampling a large number of values ($10^6$) of the stochastic process at random times. Finally, we compare the marginal density for $Y(t)$ with the analytic steady-state solution of the one-dimensional equation \eqref{eq:den1}. To quantify the differences, we used  the $L^1,L^2$ and $L^{\infty}$ norms (the parameter values are taken such that the asymptotic density is bounded).

Results are shown in Figures \ref{fig:1}-\ref{fig:1_err}. First, Figure \ref{fig:1} shows that as $\gamma_1$ increased, the marginal steady-state distribution approaches  the analytical limit. Differences between the distributions are quantified in Figure \ref{fig:1_err}, where we show norm differences between the numerical and analytic distributions. We also show the behaviour of the moments. Notice that the marginal moment of $Y$ approaches the analytic moment of the one-dimensional stochastic process as $\gamma_1\to\infty$. Also, we verify the predicted behaviour of the moment involving the first variable $X$, $\mu_k$ and $\nu_k$ for $k=1,2$, as in Proposition \ref{prop:moment}. Results show good agreement with our theoretical predictions.

\begin{figure}
\centering
  \includegraphics[width=\columnwidth]{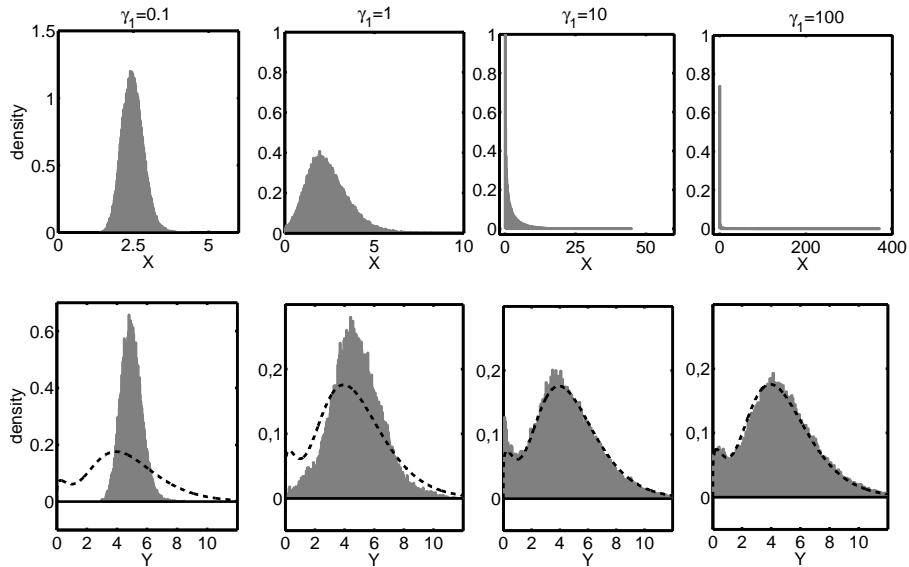}
\caption{Adiabatic reduction with the scaling (S2). Upper panels show the histograms for the first variable $X$. Bottom panels show the histograms for the second variable $Y$. Dashed lines are obtained from the one-dimensional equation \eqref{eq:den1}. Functions $\varphi(Y)$ and $h(\Delta Y)$ are given by Remark \ref{rem:hill}, and parameters used are $\varphi_0=5$, $\gamma_2=1$, $\lambda_2=2$, $K=1$, $A=4$, $B=1$, $n=4$, $b=\gamma_1/2$ and, from left to right, $\gamma_1=0.1,1,10,100$. }
\label{fig:1}
\end{figure}

\begin{figure}
\centering
  \includegraphics[width=\columnwidth]{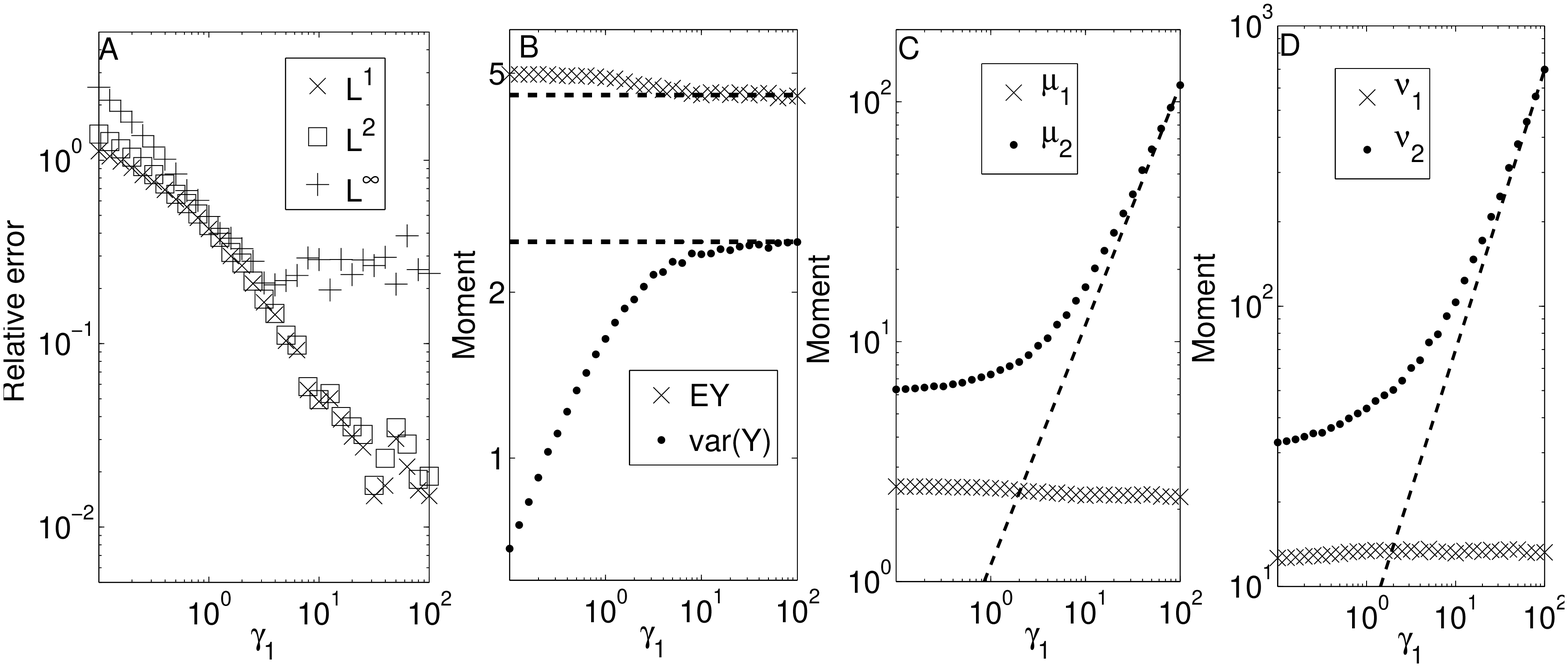}
\caption{Adiabatic reduction with the scaling (S2). (A) The norm differences between the numerical marginal density of $Y(t)$ and the analytic steady-state solution of the one-dimensional equation \eqref{eq:den1}. Results for classical $L^1,L^2$ and $L^{\infty}$ norms are shown, as indicated in the legend. (B) Asymptotic moment values of the second variable $Y$, as indicated on the legend. Dashed lines are obtained by the analytical asymptotic moment values obtained from the one-dimensional equation \eqref{eq:den1}. (C) The moments $\mu_1$ and $\mu_2$ as functions of $\gamma_1$.   (D) The moments $\nu_1$ and $\nu_2$ as functions of $\gamma_1$. In (C) and (D), the dashed lines have a slope of $+1$. Parameters used are same as in Figure \ref{fig:1}.}
\label{fig:1_err}
\end{figure}

\section{Summary}

We have considered adiabatic reduction in a model of single gene expression with auto-regulation that is mathematically described by a jump Markov process  \eqref{eq:gene1}-\eqref{eq:gene2}. If mRNA degradation is a fast process, {\it i.e.}, $\gamma_1\gg \gamma_2$, we derived reduced forms of the governing  equations under the three scaling situations so that the stationary protein level remains fixed when $\gamma_1\to\infty$: (1) If the promoter activation/deactivation is also a fast process, then the protein concentration dynamics can be approximated by a deterministic ordinary differential equation \eqref{eq:odey0}, and the mRNA concentration is approximately given by $X = b\varphi(Y)/\gamma_1$. (2) If either the transcription or the translation is a fast process, then the protein concentration dynamics can be approximated by a single stochastic differential equation with Markov jump process \eqref{eq:gene2_reduce}. We expect that these results may be generalized to justify adiabatic reduction methods in more general stochastic hybrid systems of gene regulation network dynamics.

\section*{Acknowledgments}
This work was supported by the Natural Sciences and Engineering Research Council (NSERC, Canada),  the Mathematics of Information Technology and Complex Systems (MITACS, Canada), and the National Natural Science
Foundation of China (NSFC 11272169, China), and carried out in Montr\'{e}al, Lyon and Beijing. We thank our colleague M. Tyran-Kami\'{n}ska for valuable discussions.

\bibliographystyle{spmpsci}      
\bibliography{FastSlowVar_Jan31_MCM}

\end{document}